\documentclass{amsart}  

\usepackage{amsmath,amsthm,amsfonts, eucal, amssymb}
\usepackage[english]{babel}
\usepackage{graphicx,pdflscape}
\usepackage{layout}
\usepackage[all]{xy}
\usepackage{hyperref}
\usepackage{color}
\usepackage[font=small,labelsep=none]{caption}
\usepackage{subcaption}
\usepackage{pst-all}

\usepackage{graphicx}

\title{Fiber surfaces from alternating states}

\author{Darlan Gir\~ao}

\address{Department of Mathematics, 
Universidade Federal do Cear\'a, \newline Av. Humberto Monte, s/n, Campus do Pici, Bloco 914,  
\newline 60455-760, Fortaleza -- CE, Brasil }
\email{dgirao@mat.ufc.br}

\author{Jo\~ao M. Nogueira}

\address{CMUC, Department of Mathematics, University of Coimbra,\newline
Apartado 3008 EC Santa Cruz, 3001-454 Coimbra, Portugal}
\email{nogueira@mat.uc.pt}

\author{Ant\'onio Salgueiro}

\address{CMUC, Department of Mathematics, University of Coimbra,\newline
Apartado 3008 EC Santa Cruz, 3001-454 Coimbra, Portugal}
\email{ams@mat.uc.pt}

\thanks{The first author was partially supported by FUNCAP (project PJP-0072-00013.01.00/12) and CNPq (project Universal).\\
This work was partially developed while the second author was visiting the Department of Mathematics at Federal University of Cear\'a. He is thankful for the support by the host institution on this visit and also for the warm hospitality.\\
The second and third authors were partially supported by the Centro de Matem\'{a}tica da
Universidade de Coimbra (CMUC), funded by the European Regional
Development Fund through the program COMPETE and by the Portuguese
Government through the FCT - Funda\c{c}\~{a}o para a Ci\^{e}ncia e a Tecnologia
under the project PEst-C/MAT/UI0324/2011.}


\newtheorem{thm}{Theorem}
\newtheorem{lem}
{Lemma}          
\newtheorem{cor}
{Corollary}
\theoremstyle{definition}
\newtheorem{defn}
{Definition}    
\newtheorem{rem}{Remark}            
\begin{document}

\begin{abstract}    
In this paper we define alternating Kauffman states of links and we characterize when the induced state surface is a fiber. In addition, we give a different proof of a similar theorem of Futer, Kalfagianni and Purcell on homogeneous states.
\end{abstract}

\maketitle


\section{Introduction}

Let $K$ be a link in $S^3$.    We say  the link $K$ is \textit{fibered} if $S^3-K$   has the structure of a surface bundle over the circle, i.e., if  there exists a Seifert surface $S$ such that $S^3-K\cong (S\times[0,1])/\phi$, where $\phi$ is a homeomorphism of $S$. In  this case we abuse terminology and say \textit{$S$ is a fiber for $K$}.  The study of the  fibration of link complements has been a very active line of research in low dimensional topology. In the next two paragraphs we highlight some of the work in this area.

In the  early 60's Murasugi \cite{Mu} proved that an alternating link is fibered if and only if its reduced Alexander polynomial is monic. Stallings \cite{Stallings} proved that a link $K$ is fibered if and only if $\pi_1(S^3-K)$ contains a finitely generated normal subgroup whose quotient is $\mathbb{Z}$. Stallings' result is very general, but  hard to verify, even if we restrict to particular families of links. In \cite{Ha} Harer showed how to construct all fibered knots and links using \textit{Stallings' operations} introduced in \cite{Stallings}. However,  deciding whether or not a link $K$  is fibered is, in general, a hard problem.  Goodman--Tavares \cite{GT} showed that under simple conditions imposed on certain Seifert surfaces for pretzel links, it is possible to decide whether or not these surfaces are fibers.   In \cite{Gabai} Gabai proved that if a Seifert surface $S$ can be decomposed as the \textit{Murasugi sum} of surfaces $S_1,...,S_n$, then $S$ is a fiber if and only if each of the surfaces $S_i$ is a fiber (refer to theorem \ref{Gabai}). 

 Very recently Futer--Kalfagianni--Purcell \cite{FKP} introduced a new method for deciding whether some Seifert surface are fibers. From a particular diagram of the link, they construct an associated surface (called \textit{state surface}) and a certain graph. If the state is \textit{homogeneous}, they show that this surface is a fiber if and only if the corresponding graph is a tree (Theorem \ref{futer} below). Later, Futer \cite{Futer} gave a different, much simpler  proof of this result. Based on the work of Gabai \cite{Gabai} and Stallings \cite{Stallings}, the first author \cite{Girao} studied  fibration of state surfaces of augmented links.  This paper is concerned with the study of another class of state surfaces, which we now describe.

Given a diagram $D$ of a link $L$ we can construct a collection of disjoint disks connected by a twisted band at each crossing. We thus obtain a surface whose boundary is the link $L$. The disks and bands are defined by how we split the crossings in the diagram of $L$. At each crossing there are two choices of \textit{resolutions} for the split: an $A$-resolution or a $B$-resolution, as presented in Figure \ref{Figure:Resolutions}.\\   

\begin{figure}[ht]
\begin{center}
\includegraphics{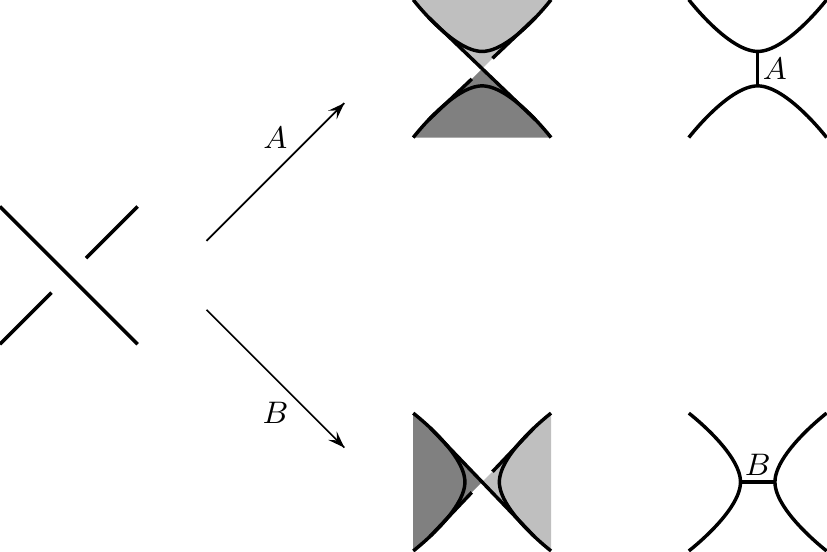}
\end{center}
\caption{: The two choices of resolutions for the split of a crossing.}
\label{Figure:Resolutions}
\end{figure}

A \textit{Kauffman state} $\sigma$ of a link diagram $D$ is a choice of resolution for each crossing of $D$. The resulting surface $S_\sigma$ is called the \textit{state surface} of $\sigma$. The boundaries of the disks induce a decomposition of the plane into connected components that we call \textit{regions}. The well known Seifert surface of an oriented diagram of a link is a particular case of a state surface, where the resolution of each crossing is defined by the orientation. It has been an interest of research to identify fibered knots and their fibers. We are interested in understanding when a state surface is a fiber.  In the work of  Futer, Kalfagianni and Purcell  \cite{FKP} it was studied for \textit{homogeneous states}, that is when all resolutions of the diagram in each region are the same (see Theorem \ref{futer}). Futer \cite{Futer} gave an alternate, much simpler proof of this theorem.   In this paper we provide a new approach for this theorem and we prove a similar result for a different type of Kauffman states, as in the next definition.

\begin{defn}
A Kauffman state $\sigma$ is said to be \textit{alternating} when for each circle defined by $\sigma$, with a choice of orientation on its boundary, if two consecutive crossings attached to it in the same region have the same resolution then they are adjacent to the same circles defined by $\sigma$.
\end{defn}

Before we present our main result, we associate two graphs to each state of a link diagram. The \textit{state graph} $G_{\sigma}$ has one vertex for each disk and one edge for each band defined by $\sigma$. We label the edges by the resolution of the respective crossings. The \textit{reduced graph} $G'_{\sigma}$ is obtained from $G_{\sigma}$ by eliminating duplicated edges, with the same label, between two vertices. From the state surface $S_{\sigma}$ we define also a \textit{reduced surface} $S'_{\sigma}$ by cutting duplicated bands with the same label attached to the same pair of disks.  We note that the graphs $G_\sigma$ and $G'_\sigma$ are not abstract graphs but instead they are embedded in the surfaces $S_\sigma$ and $S'_\sigma$ as their spines. An \textit{inner cycle}, of the state graph or a reduced version of it, is an innermost cycle in a certain region. Our main result is the following.

\begin{thm}\label{theorem:alternating}
Let $\sigma$ be an alternating state of a link diagram $D_L$. Then $E(L)$ fibers over the circle with fiber $S_{\sigma}$ if and only if the reduced graph $G'_{\sigma}$ is a tree.
\end{thm}

The next examples illustrate that the classes of link diagrams in theorems \ref{theorem:alternating} and \ref{futer} are distinct. Certain states can be both homogeneous and alternating, as for example the Seifert state of the Figure eight knot as in Figure \ref{Figure:Figure_eight_knot}.

\begin{figure}[ht]
\begin{center}
\includegraphics{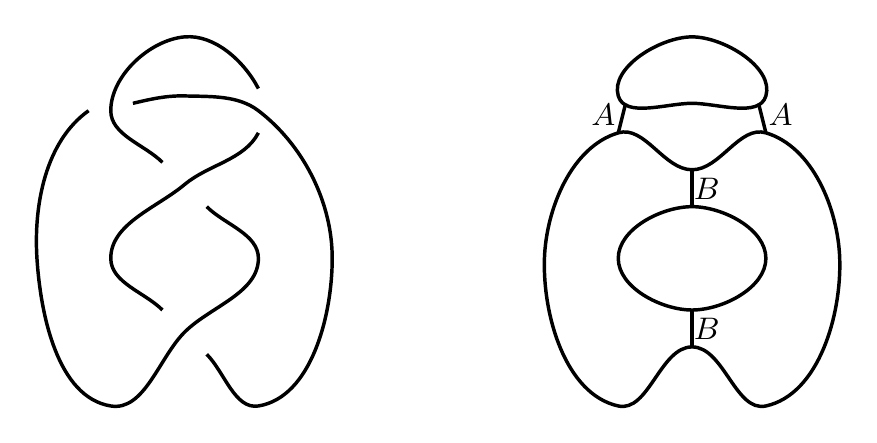}
\end{center}
\caption{: The Seifert state of this Figure eight knot diagram is a fiber by theorems \ref{theorem:alternating} and \ref{futer}.}
\label{Figure:Figure_eight_knot}
\end{figure}

But in general a state isn't both homogeneous and alternating. For instance, in the Figure \ref{Figure:12crossings} the Seifert state is alternating and not homogeneous.\\

\begin{figure}[ht]
\begin{center}
\includegraphics{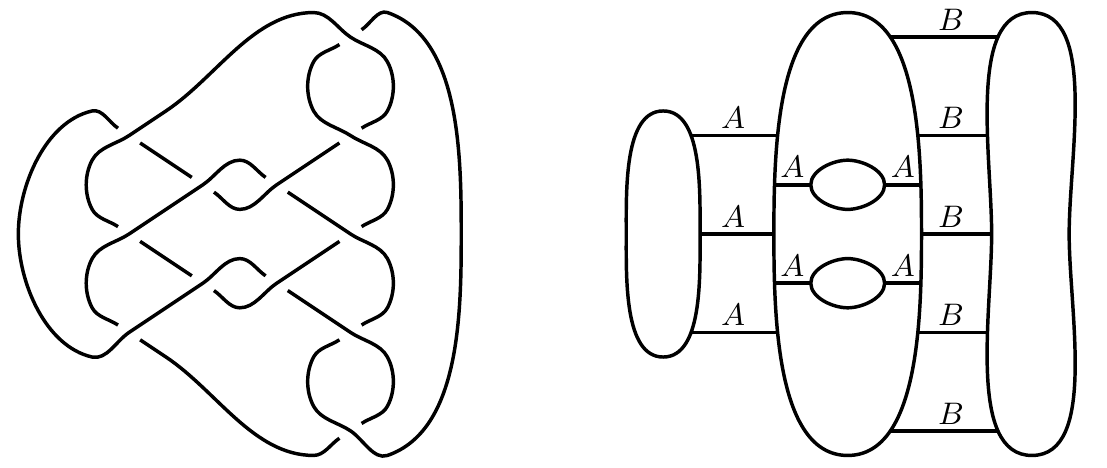}
\end{center}
\caption{: The knot $12n0328$ is prime, the Seifert state of this diagram is alternating and not homogeneous, and the corresponding state surface is a fiber by Theorem \ref{theorem:alternating}.}
\label{Figure:12crossings}
\end{figure}
Furthermore, in the example shown in Figure \ref{Figure:double_trefoil}, the Seifert state is homogeneous but not alternating.

\begin{figure}[ht]
\begin{center}
\includegraphics{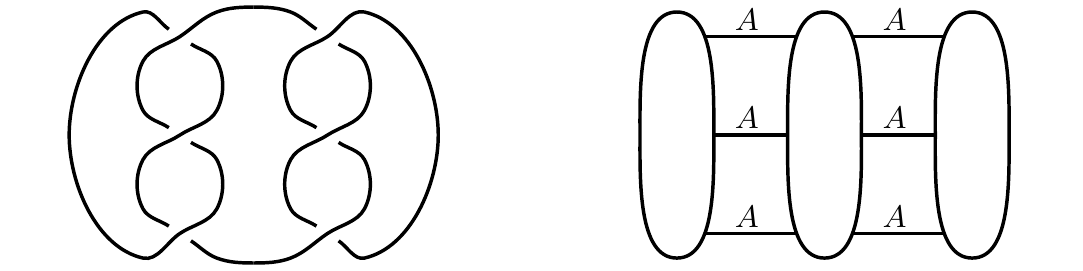}
\end{center}
\caption{: The Seifert state of this granny knot diagram is homogeneous and not alternating, and the corresponding state surface is a fiber by Theorem \ref{futer}.}
\label{Figure:double_trefoil}
\end{figure}

The reduced graph of the state in the examples of Figures \ref{Figure:Figure_eight_knot}, \ref{Figure:12crossings}
 and \ref{Figure:double_trefoil} is a tree, so in these particular cases the state surface is a fiber. We notice that if $G_{\sigma}$ has edges with different labels between the same pair of vertices then $G'_{\sigma}$ is not a tree and, by Theorem \ref{theorem:alternating}, $S_{\sigma}$ is not a fiber.\\

In section \ref{section:alternating} we prove this theorem using Murasugi sums and results of Gabai on knot fibration. In section \ref{homogeneous} we give a different, homological proof, of the following theorem of Futer, Kalfagianni and Purcell \cite{FKP} on homogeneous states. The techniques we use in our proof are similar to the ones in the paper \cite{Girao} by the first author, where he studies the fibration of augmented link complements.

\begin{thm}\label{futer}
Let $\sigma$ be a homogeneous state of a link diagram $D_L$. Then $E(L)$ fibers over the circle with fiber $S_{\sigma}$ if and only if the reduced graph $G'_{\sigma}$ is a tree.
\end{thm}


\section{Fibers from alternating states}\label{section:alternating}

For this section we use a specific concept of graph decomposition: We say that two vertices, $v$ and $w$, \textit{decompose} a graph $G$ into components $G_1,\ldots, G_k$ if \linebreak$G=G_1\cup\cdots \cup G_k$ and $G_i\cap G_j\subseteq \{v, w\}$, for $i\neq j$. We also make use of the the following theorem by Gabai \cite{Gabai} on Murasugi sum and knot fibration.

\begin{defn}
We say that the oriented surface $T$ in $S^3$ with boundary $L$ is the Murasugi sum of the two oriented surfaces $T_1$ and $T_2$ with boundaries $L_1$ and $L_2$ if there exists a $2$-sphere $S$ in $S^3$  bounding  balls $B_1$ and $B_2$ with $T_i\subset B_i$ for $i=1,2$, such that $T=T_1\cup T_2$ and $T_1\cap T_2=D$ where $D$ is a $2n$-sided disk contained in $S$.  
\end{defn} 

\begin{thm}[Gabai]\label{Gabai}
Let $T\subset S^3$, with $\partial T=L$, be a Murasugi sum of oriented surfaces $T_i\subset S^3$, with $\partial T_i=L_i$, for $i=1, 2$. Then $S^3-L$ is fibered with fiber $T$ if and only if $S^3-L_i$ is fibered with fiber $T_i$ for $i=1, 2$.
\end{thm}

With the following lemma we are able to prove that we neither lose fibration information by working with the reduced state graph nor with graph decomposition. 

\begin{lem}\label{murasugi decomposition}
Let $G_{\sigma}$ be a state graph and suppose there are two vertices, $v$ and $w$, adjacent by the edge $X$, that decompose $G_\sigma$ into connected components $X$, $H_1$, $H_2$, $\ldots$, $H_k$. (See Figure \ref{Figure:reduction}.) Consider also the state surface $S_i$ induced by $\sigma$ and the subgraph $X\cup H_i$ of $G_\sigma$, $i=1,\ldots, k$. Then, $S_\sigma$ is a fiber if and only if each surface $S_1,\ldots, S_k$ is a fiber with respect to its boundary.

\begin{figure}[ht]
\begin{center}
\includegraphics{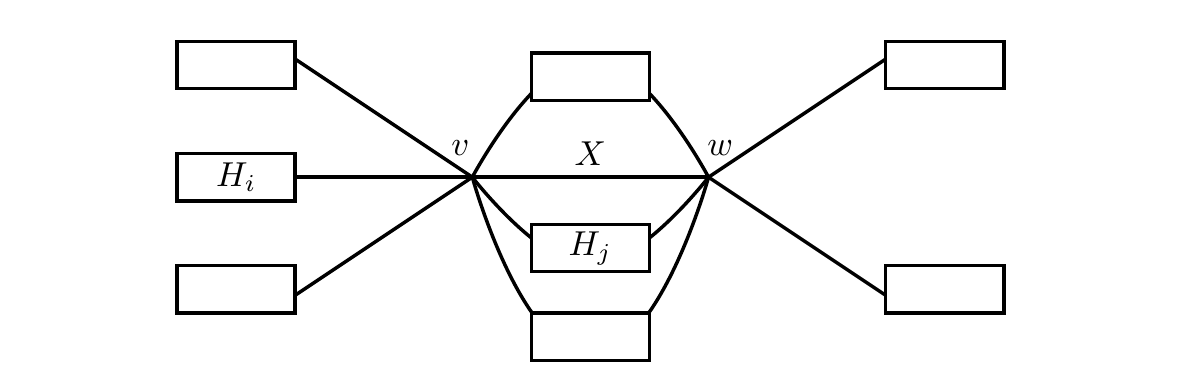}
\end{center}
\caption{: Representation of the decomposition of $G_\sigma$  by $v\cup w$.}
\label{Figure:reduction}
\end{figure}

\end{lem}
\begin{proof}
We start by proving that $S_\sigma$ is a Murasugi sum of the surfaces $S_1,\ldots, S_k$.  
Consider one of the connected components $H_l$. If $H_l$ contains only one of the vertices $v$ or $w$, then using the disk associated to this vertex and $X$ we can decompose $S_l$ from $S_\sigma$ by a Murasugi sum. (See Figure \ref{Figure:1 component}.) Notice that $S_l$ is also the state surface of $H_l$, since $X$ contains a terminal vertex in $X\cup H_l$.
\begin{figure}[ht]
\begin{center}
\includegraphics{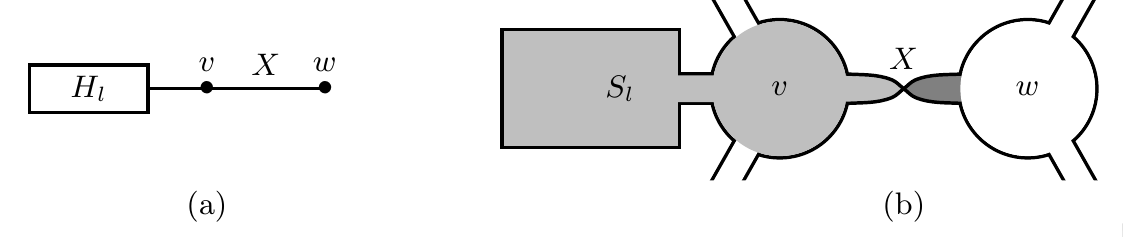}
\end{center}

\

\begin{center}
\includegraphics{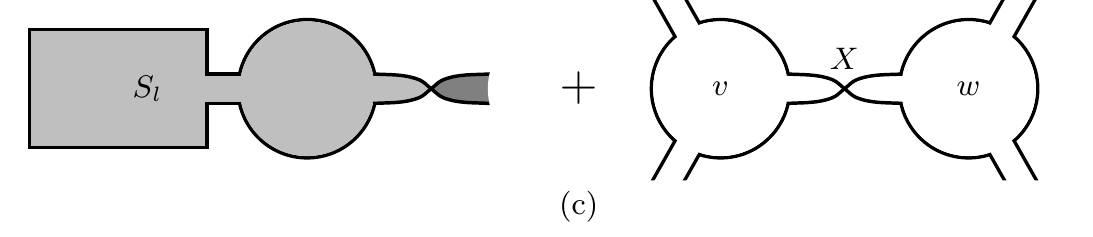}
\end{center}
\caption{: When $H_l$ is adjacent to only $v$, (a), there is a decomposition of $S_\sigma$, (b),  by $v\cup w$ as a Murasugi sum, (c).}
\label{Figure:1 component}
\end{figure}

Let us assume now that $H_l$ contains $v\cup w$. Suppose, without loss of generality, that $H_l$ is innermost with respect to $X$, i.e., there is no other component between $H_l$ and $X$ in the state graph. We can decompose $S_l$ from $S_\sigma$ by a Murasugi sum as depicted in Figure \ref{Figure:2 components}.

\begin{figure}[ht]
\begin{center}
\includegraphics{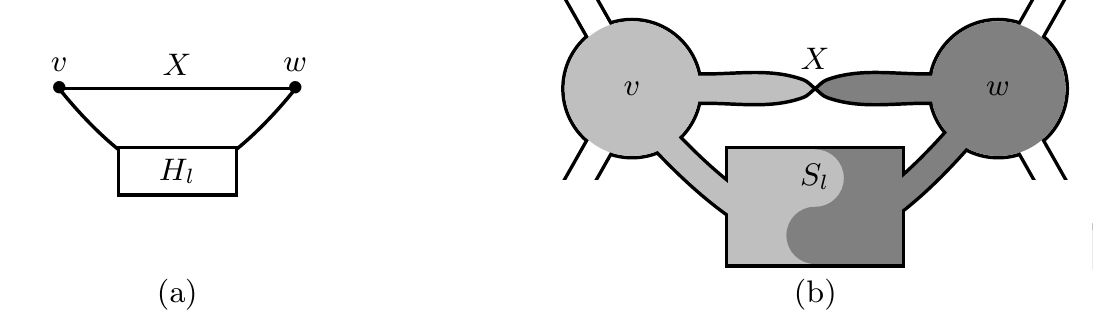}
\
\includegraphics{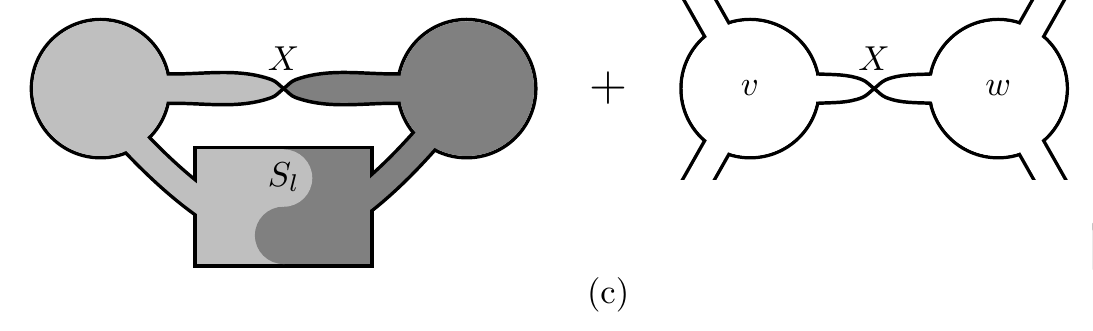}
\end{center}
\caption{: When $H_l$ is adjacent to both $v$ and $w$, (a), there is a decomposition of $S_\sigma$, (b),  by $v\cup w$ as a Murasugi sum, (c).}
\label{Figure:2 components}
\end{figure}

Repeating this procedure with subsequent innermost components we obtain the claimed Murasugi sum decomposition. Therefore, by Theorem \ref{Gabai}, $S_\sigma$ is a fiber if and only if each surface $S_l,\ldots, S_k$ is a fiber with respect to its boundary.
\end{proof}

A particular case of this lemma is when two vertices $v$ and $w$  are adjacent by multiple edges. Take two such edges and suppose $X$ and $Y$ are their labels (we also represent the edges by these letters).  Decomposing the graph $G_{\sigma}$ as in  the lemma, one of the components obtained corresponds to  the edge  $Y$. It is not hard to see that the state surface induced by the subgraph $X\cup Y$ is either a Hopf band (when edges have the same label) or an untwisted annulus (when edges have different labels).  It is well known that the Hopf band is a fiber for the Hopf link in its boundary, and that the untwisted annulus is not a fiber for the unlink in its boundary. For example, a straightforward proof of these facts follows from  Theorem \ref{stallings}.  This observation lead us to the following corollary.

\begin{cor}\label{duplicated edges}
Let $G_{\sigma}$ be a state graph and suppose there are  vertices $v$ and $w$ adjacent by two edges $X$ and $Y$.  If the edges have different labels then the surface $S_\sigma$ is not a fiber. If the edges have the same label then $S_\sigma$ is a fiber if and only if the state surface induced by the subgraph obtained by removing the edge $Y$ is a fiber.
\end{cor}

\begin{proof}
Decomposing the edge $Y$ from the graph as in   Lemma \ref{murasugi decomposition},  the surface induced by $X\cup Y$ is either an untwisted annulus or a Hopf band.  By the observation above and Theorem \ref{Gabai}, in the former case  the surface $S_\sigma$ is not a fiber; in the latter case $S_\sigma$ is a fiber if and only if the remaining Murasugi summands are fibers, that is,  if  the surface induced  by the subgraph obtained by removing the edge $Y$ is a fiber.
\end{proof}

\begin{rem}
In light of Corollary \ref{duplicated edges} we assume from now on that the state graph $G_\sigma$ has no edges with different labels adjacent to the same pair of vertices.
\end{rem}

Corollary \ref{duplicated edges} explains why we do not loose fibering information by passing to the reduced graph $G_{\sigma}'$. 

\begin{cor}\label{reduced surface}
Let $L$ be a link and $\sigma$ a state for a diagram $D_L$ of $L$. Let $S'_{\sigma}$ be the state surface associated to the the reduced graph $G'_{\sigma}$ and let  $L'$ the boundary of  $S'_{\sigma}$. Then   the link $L'$ is fibered by $S'_{\sigma}$ if and only if $L$ is fibered by $S_{\sigma}$.
\end{cor}

\begin{proof}
This is a immediate consequence of Corollary \ref{duplicated edges}.
\end{proof}

\begin{lem}\label{X cycle}
Let $G_{\sigma}$ be a state graph and suppose there are two vertices, $v$ and $w$, that decompose $G_\sigma$ into two connected components $X$ and $Y$, and there is an alternating path $\alpha$ from $v$ to $w$, in $Y$, that together with $X$ define an inner cycle. (See Figure \ref{Figure: alternating cycle}.) Consider also the state surface $S_y$ induced by $\sigma$ and $Y$, and the state surface $S_x$ induced by $\sigma$ and $X\cup\alpha$. Then, $S_\sigma$ is a fiber if and only if each surface, $S_x$ and $S_y$, is a fiber with respect to its boundary. 

\begin{figure}[ht]
\begin{center}
\includegraphics{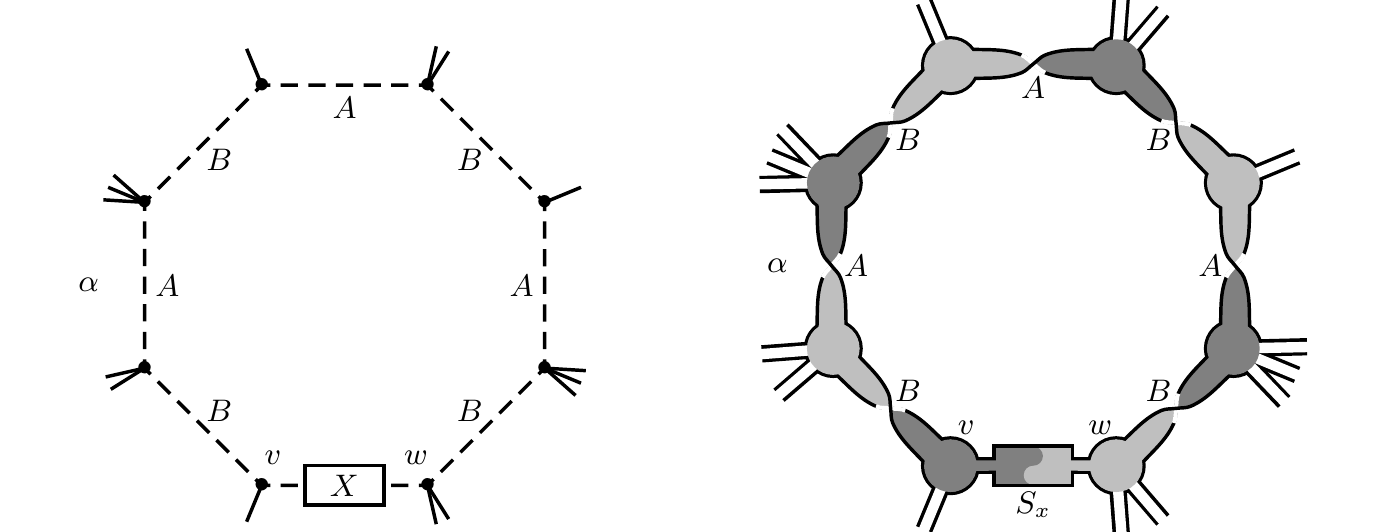}
\end{center}
\caption{: An alternating path $\alpha$ in that together with $X$ define an inner cycle.}
\label{Figure: alternating cycle}
\end{figure}

\end{lem}
\begin{proof}
 Since $X\cup \alpha$ defines an inner cycle and $\alpha$ is alternating, with respect to the labels, then there is a ball $Q$ intersecting $S_\sigma$ at $S_x$ with the band associated with $\alpha$ in $\partial Q$. In this way, we can decompose $S_\sigma$ as a Murasugi sum of $S_x$ and $S_y$, as depicted in Figure \ref{Figure: alternating cycle decomposition}.\\

\begin{figure}[ht]
\begin{center}
\includegraphics{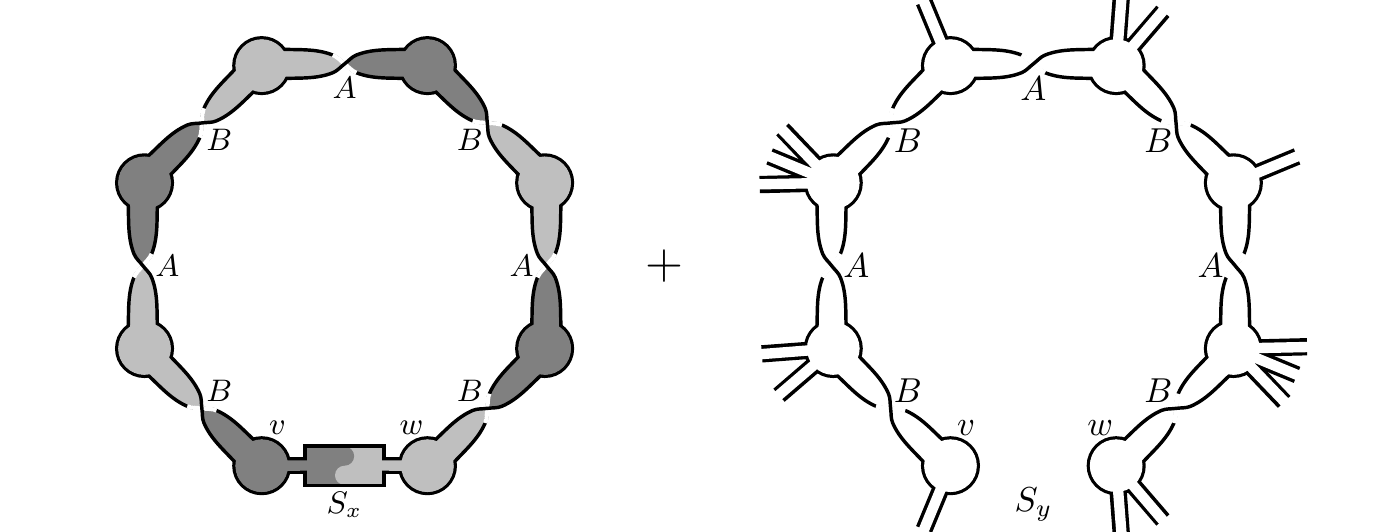}
\end{center}
\caption{: Decomposition of $S_\sigma$ by $\alpha$ as a Murasugi sum of $S_x$ and $S_y$.}
\label{Figure: alternating cycle decomposition}
\end{figure}

From the result of Gabai and this Murasugi sum we have the statement of the lemma.
\end{proof}

\begin{lem}\label{alternating inner cycle}
If the state graph $G_{\sigma}$ has an inner cycle that is alternating with respect to the labels $A$ or $B$ then $S_{\sigma}$ is not a fiber of $L$.
\end{lem}
\begin{proof}
Consider an inner cycle $\gamma$ of $G_\sigma$. In Lemma \ref{X cycle}, let $X$ be one edge of $\gamma$ and $\alpha$ the remaining edges. Then, $S_\sigma$ is a fiber if and only if $S_x$ and $S_y$ are fibers. Since $\gamma$ is alternating then $S_x$ is an annulus, which is not a fiber of its boundary. Hence, $S_\sigma$ is not a fiber of $L$.  
\end{proof}

\begin{proof}[Proof of Theorem \ref{theorem:alternating}]
We start by observing that if $G'_\sigma$ is a tree then $S'_\sigma$ is a disk, and hence a fiber of $L'$. Therefore, by Corollary \ref{reduced surface}, $L$ is fibered by $S_\sigma$.\\
Suppose now that $G'_\sigma$ has a cycle. Then $G_\sigma$ also has a cycle. Consider an inner cycle $\alpha$ of $G_\sigma$.  Suppose  there is a tree component to the interior of $\alpha$ at the common vertex. Then, using Lemma \ref{murasugi decomposition}, we may decompose this tree. The surface induced by this tree is a disk, and hence a fiber for its boundary circle. Therefore, we may assume the following: consecutive edges in $\alpha$ are also consecutive in the common vertex, i.e, there are no edges to the interior of $\alpha$ between them.     Given any two such edges in $\alpha$, since $\sigma$ is alternating they have  different labels. Hence, $\alpha$ is alternating. Consequently, by Lemma \ref{alternating inner cycle} the state surface $S_\sigma$ is not a fiber of $L$.
\end{proof}

\section{A new proof of Theorem \ref{futer}}\label{homogeneous}

In this section we present a different proof of Theorem \ref{futer}. This theorem first appeared in \cite{FKP} (Theorem 5.21), but the proof presented there consists of a detailed study of polyhedral decompositions of $S^3-S_{\sigma}$. In \cite{Futer} a much simpler proof is given: it is proved inductively via Murasugi sums together with Theorem \ref{Gabai} to deduce fibering information. Some of these ideas were also independently used in the work of the first author \cite{Girao} and in the previous section. The proof we present is  a consequence of Stallings' fibration criteria \cite{Stallings}.

\begin{thm}[Stallings]\label{stallings}
Let $T\subset S^3$ be a compact, connected, oriented surface with nonempty boundary $\partial T$. Let $T\times[-1,1]$ be a regular neigborhood of $T$ and let $T^+=T\times\{1\}\subset S^3-T$. Let $f=\varphi|_{T}$, where $\varphi:T\times[-1,1]\longrightarrow T^+$ is the projection map. Then $T$ is a fiber for the link $\partial T$ if and only if  the induced map $f_*:\pi_1(T)\longrightarrow\pi_1(S^3-T)$ is an isomorphism.   
\end{thm}

We describe the induced map in the case $T$ is the state surface associated to the reduced graph of a homogeneous link diagram. We will see that when $G'_{\sigma}$ is a tree, the reduced surface $S'_{\sigma}$ is a disk and  the map $f_*$ is trivial, as desired. When $G'_{\sigma}$ has cycles, we show that the map $f_*$ cannot be an isomorphism by showing that the corresponding map on first homology is not an isomorphism. 

By using Lemma \ref{murasugi decomposition}, we may  decompose the reduced state graph $G'_{\sigma}$ associated with the homogeneous link diagram along cut vertices. This has also been observed in Lemma 3 of \cite{Futer}, where he proves that the reduced graph has no cut vertices if and only if it is  an all-$A$ or all-$B$ state. Thus, we only need to prove this result for  all-$A$ or all-$B$ states. We provide the proof for the case of an all-$A$ state, the other case being similar. 

First note that in the absence of cut vertices in the graph $G'_{\sigma}$, the surface $S'_{\sigma}$ is a checkerboard surface.  If the graph $G'_{\sigma}$ is a tree, then the surface $S'_{\sigma}$ is a disk. Hence $S'_\sigma$ is a fiber, and by Corollary \ref{reduced surface} the surface $S_\sigma$ is also a fiber.

Suppose now that  $G'_{\sigma}$ is not a tree, i.e., that it has cycles. We will prove that this contradicts Stalling's theorem. First note that the fundamental group of the surface $S'_{\sigma}$ is free. Consider the inner cycles $\alpha_{1},...,\alpha_{n}$ in $G'_{\sigma}$ oriented in the counter-clockwise direction. Since $S_\sigma$ is a fiber, it is orientable, hence $S'_{\sigma}$ is also orientable and we choose a base point $a$ of $\pi_1(S'_\sigma)$ such that, when seen  from above the projecting plane,  we see the base point $a$ in the ``$+$'' side of $S'_{\sigma}$.  Finally, add  arcs $h_{1},...,h_{n}$  from $a$ to each of the inner cycles above. This gives loops $\beta_{i}=h_{i}\alpha_{i}h_{i}^{-1}$,  based at $a$.  This set of based loops corresponds to a generating set  for $\pi_1(S'_{\sigma})$. These generators will  be denoted by $u_{1},...,u_{n}$.

Since the surface $S'_{\sigma}$ is a checkerboard surface, its complement $S^3-S'_{\sigma}$ also has a free fundamental group. We now describe a generating set for this group.  There are two types of white regions  in the projecting plane: one unbounded region and $n$ bounded ones, which correspond to the inner cycles of $G'_{\sigma}$. Let $C_0$ denote the unbounded white region determined by  $S'_{\sigma}$ and let $A_i$ denote a white region determined by the inner cycle $\alpha_i$.  Let $\gamma_{i}\subset S^3-S'_{\sigma}$ be a semi-circle   with one endpoint in $C_0$ and the other in $A_i$,  lying under the projecting plane.  Let ${f:S'_{\sigma}\longrightarrow S^3-S'_{\sigma}}$ be the function described in Theorem \ref{stallings}. Associated to each region $A_i$  we construct a simple closed curve by connecting the endpoints of the arc $\gamma_{i}$ to the point $f(a)$ by straight line segments. Each of these curves is oriented so that, starting at $f(a)$, we move along the line segment connecting $f(a)$ to the endpoint of $\gamma_{i}$ in $A_i$, then move along $\gamma_{i}$ to the second endpoint and then back to $f(a)$ through the second line segment.  We have built  loops with base point $f(a)$ corresponding to a set of generators for $\pi_1(S^3-S'_{\sigma})$.   These generators are denoted   by  $x_{1},...,x_{n}$, according to the label of region they cross.

Let $S^{'+}_{\sigma}$ be the copy of $S'_{\sigma}$ in $S^3-S'_{\sigma}$ parallel to $S'_{\sigma}$,  obtained  from $S'_{\sigma}$ by  pushing it  in the  ``$+$'' direction.  This is formally defined by the map  ${f:S'_{\sigma}\longrightarrow S^3-S'_{\sigma}}$ described in Theorem \ref{stallings}. The induced map $f_*$ can be described by determining the image of each generator $u_i\in\pi_1(S'_{\sigma})$. We write $f_*(u_i)$ as a word on the generators $x_1, ..., x_n$, given by the image the loop $\beta_i=h_{i}\alpha_{i}h_{i}^{-1}$: 
$$f_*(u_i)=w_{h_i}w_{\alpha_i}w_{h_i}^{-1}$$
where $w_{h_i}$ is the word on the letters $x_1, ..., x_n$ given by the image of the arcs $h_i$ under the map $f$. The word  $w_{\alpha_i}$ is obtained by the image of the cycle  $\alpha_i$ as follows.  Suppose that $\alpha_i$ and $\alpha_j$ have a common edge. Vertices are labeled ``$+$'' or ``$-$'', depending on the side of the surface they lie.  We have two possibilities: 
\begin{itemize}
\item[Case 1.] The orientation induced on the edge by $\alpha_i$ is from a ``$+$'' vertex to a ``$-$'' vertex. In this case we write the letter $x_i$. (See Figure \ref{map_f} left.)
\item[Case 2.] The orientation induced on the edge by $\alpha_i$ is from a ``$-$'' vertex to a ``$+$'' vertex. In this case we write the letter $x_j^{-1}$. (See Figure \ref{map_f} right.)
\end{itemize}


\begin{figure}[ht]
\begin{center}
\includegraphics{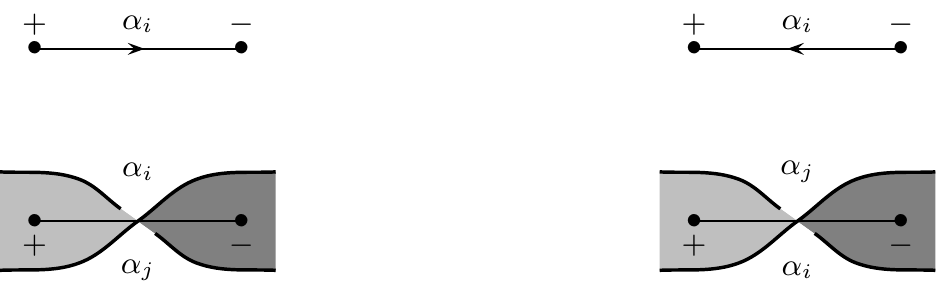}
\end{center}
\caption{: Case 1 (left); case 2 (right).}
\label{map_f}
\end{figure}

\begin{rem}
It is important to notice the inner cycle $\alpha_i$ may share an edge with the unbounded region $C_0$. If this is the case, in 2 above, we write no letters corresponding to this edge. 
\end{rem}

   \begin{rem}\label{Remark:opposite signs}
   Observe that the loops $\alpha_i$ and $\alpha_j$ induce reverse orientations on  the edges they share. Therefore, when we write the letters corresponding to the loop $\alpha_j$, the letter corresponding to this edge is the same letter as $\alpha_i$, with opposite sign, i.e., either $x_i^{-1}$ or $x_j$.  This is illustrated in Figure \ref{map_f}. 
\end{rem}

Now we consider the map $\bar{f}_*:H_1(S'_{\sigma})\longrightarrow H_1(S^3-S'_{\sigma})$ induced on homology by $f_*$. Denote by $\bar{u}_{1},...,\bar{u}_{n}$ the generators of $H_1(S'_{\sigma})$, corresponding to the  generators of $\pi_1(S'_{\sigma})$. The generators of $H_1(S^3-S'_{\sigma})$ are defined  similarly and denoted $\bar{x}_{1},...,\bar{x}_{n}$.

The map $\bar{f}_*$ is given by a $n \times n$  matrix $\mathcal{A}=[a_{ij}]$, where the $i$-th column is the vector $\bar{f}(\bar{u}_i)\in H_1(S^3-S'_{\sigma})$.  By the description of the map $f_*$ and the remarks above,  the matrix $\mathcal{A}$ has the following properties: 

\begin{itemize}
\item[(i)] $a_{ii}\geq 2$;
\item[(ii)] $a_{ii}\geq\displaystyle\sum_{j\neq i}|a_{ij}|$
\item[(iii)]  $a_{ii}\geq\displaystyle\sum_{j\neq i}|a_{ji}|$
\end{itemize}
(i) follows from the fact that every inner cycle in $G'_{\sigma}$ has at least 4 edges; (ii) and (iii) follow from the fact that, when we go through the cycle $\alpha_i$, at every other edge we write the letter $x_i$ and at the remaining edges we write one of the other letters $x_j$ or write no letters (as in Remark \ref{Remark:opposite signs}).   

To prove that the map $f_*$ is not an isomorphism if $G'_{\sigma}$ is not a tree (i.e.,  has cycles), it suffices to prove the matrix $\mathcal{A}$ is not invertible over $\mathbb{Z}$. This is straightforward by the following theorem.

\begin{thm}
Let $\mathcal{A}=[a_{ij}]$ be such that $a_{ii}\geq \max(2,\sum\limits_{j\neq i}|a_{ij}|),\forall i\in\{1,\ldots,n\}$. If $\det(\mathcal{A})\neq 0$, then $\det(\mathcal{A})\geq 2$ and this inequality is sharp.
\end{thm}

\begin{proof}
We will prove the theorem by induction on $n$.

For $n=1$, $\det(\mathcal{A})=\det[a_{11}]=a_{11}\geq 2$.

Consider now any $n\in\mathbb{N}$ and suppose that the result is true for $n-1$.

Suppose $\det(\mathcal{A})\neq0$ and let $\mathcal{B}=[b_{ij}]\in M_n(\mathbb{Z})$ be the adjugate matrix of $\mathcal{A}$. Then $\mathcal{A}\mathcal{B}=(\det \mathcal{A})\mathcal{I}_n$. 

If all elements of the column $j$ of $\mathcal{B}$ have the same absolute value $b_{jj}$, then \sloppy${\det(\mathcal{A})=\sum_k a_{ij}b_{ji}}$ is a multiple of $b_{jj}\geq 2$.

If not, suppose $|b_{ij}|\geq |b_{kj}|,\forall k\in\{1,\ldots,n\}$ and $|b_{ij}|>|b_{kj}|$ for some $k$. Then
$$\left|\sum\limits_{k=1}^n a_{ik}b_{kj}\right|=\left|a_{ii}b_{ij}+\sum\limits_{k\neq i}a_{ik}b_{kj}\right|\geq |a_{ii}b_{ij}|-\sum\limits_{k\neq i}|a_{ik}b_{kj}|>|a_{ii}b_{ij}|-|a_{ii}b_{ij}|=0.$$
Since $\mathcal{A}\mathcal{B}$ is a diagonal matrix, then $i=j$. Therefore, $|b_{ii}|>|b_{ki}|,\forall k\neq i$. Furthermore, by the induction hypothesis, $b_{ii}\geq 2$. Hence
$$
\det(\mathcal{A})=\left|\sum_{k=1}^n a_{ik}b_{ki}\right|=\left|a_{ii}b_{ii}+\sum\limits_{k\neq i}a_{ik}b_{ki}\right|\geq|a_{ii}b_{ii}|-\sum\limits_{k\neq i}|a_{ik}b_{ki}|\geq 
$$
$$
\geq a_{ii}b_{ii}-a_{ii}(b_{ii}-1)=a_{ii}\geq 2.
$$

To see that the inequality is sharp, observe that the determinant of the $n\times n$ tridiagonal matrix
$$
\mathcal{A}=\left[
\begin{array}{cccccccc}
2 & 2 & 0 & 0 & 0 & \cdots & 0& 0 \\
1 & 2 & -1 & 0 & 0& \cdots & 0 &0 \\
1 & 0 & 2 & 1 & 0& \cdots & 0 & 0\\
0 & 0 & 1 & 2 & 1&\cdots & 0 & 0\\
\vdots &\vdots &\vdots &\vdots &\vdots &\ddots & \vdots& \vdots\\
0 & 0 & 0 & 0 & 0&\cdots & 2 & 1 \\
0 & 0 & 0 & 0 & 0&\cdots & 1 & 2 
\end{array}
\right]
$$
is 2, for every $n\in\mathbb{N}$.
\end{proof}

\end{document}